\documentclass{amsart}

\usepackage{amsmath}
\usepackage{amsfonts}
\usepackage{amsopn}
\usepackage{amssymb}

\DeclareMathOperator{\SL}{SL}

\DeclareMathOperator{\Lip}{Lip}

\DeclareMathOperator{\SO}{SO}

\DeclareMathOperator{\diagmat}{diag}
\DeclareMathOperator{\trace}{trace}

\newcommand{\emptyword}{\varepsilon}

\newcommand{\R}{\ensuremath{\mathbb{R}}}
\newcommand{\Q}{\ensuremath{\mathbb{Q}}}

\newcommand{\N}{\ensuremath{\mathbb{N}}}
\newcommand{\K}{\ensuremath{\mathbb{K}}}
\newcommand{\Z}{\ensuremath{\mathbb{Z}}}
\newcommand{\cE}{\ensuremath{\mathcal{E}}}
\newcommand{\cM}{\ensuremath{\mathcal{M}}}

\newcommand{\cP}{\ensuremath{\mathcal{P}}}
\newcommand{\cS}{\ensuremath{\mathcal{S}}}

\newcommand{\short}{\widehat}

\renewcommand{\setminus}{{-}}

\newtheorem{thm}{Theorem}

\newtheorem{lem}{Lemma}
\newtheorem{lemma}[lem]{Lemma}

\newtheorem{cor}[lem]{Corollary}
\newtheorem{conj}{Conjecture}
\title{A polynomial isoperimetric inequality for $\SL(n,\Z)$}
\author{Robert Young}
\address{Institut des Hautes \'Etudes Scientifiques,\\
Le Bois Marie, 35 route de Chartres, F-91440 Bures-sur-Yvette, France}
\date{\today}
\email{rjyoung@ihes.fr}

\begin{document}
\bibliographystyle{amsplain}
\begin{abstract}
  We prove that when $n\ge 5$, the Dehn function of $\SL(n,\Z)$ is at
  most quartic.  The proof involves decomposing a disc in
  $\SL(n,\R)/\SO(n)$ into a quadratic number of loops in generalized
  Siegel sets.  By mapping these loops into $\SL(n,\Z)$ and replacing
  large elementary matrices by ``shortcuts,'' we obtain words of a
  particular form, and we use combinatorial techniques to fill these
  loops.
\end{abstract}
\maketitle

\section{Introduction}
The Dehn function is a geometric invariant of a space (typically, a
riemannian manifold or a simplicial complex) which measures the
difficulty of filling closed curves with discs.  This can be made into
a group invariant by defining the Dehn function of a group to be the
Dehn function of a space on which the group acts cocompactly.  The
choice of space affects the Dehn function, but its rate of growth
depends solely on the group.

The study of Dehn functions of lattices in semisimple Lie groups is a
natural direction.  For cocompact lattices, this is straightforward;
such a lattice acts on a non-positively curved symmetric space $X$,
and this non-positive curvature gives rise to a linear or quadratic
Dehn function.  Non-cocompact lattices have more complicated behavior.
The key difference is that if the lattice is not cocompact, it acts
cocompactly on a subset of $X$ rather than the whole thing, and the
boundary of this subset may contribute to the Dehn function.

In the case that $\Gamma$ has $\Q$-rank 1, the Dehn function is
almost completely understood, and depends primarily on the $\R$-rank
of $G$.  In this case, $\Gamma$ acts cocompactly
on a space consisting of $X$ with infinitely many
disjoint horoballs removed.  When $G$ has $\R$-rank
1, the boundaries of these horoballs correspond to nilpotent groups,
and the lattice is hyperbolic relative to these nilpotent groups.  The
Dehn function of the lattice is thus equal to that of the nilpotent
groups, and Gromov showed that unless $X$ is the complex,
quaternionic, or Cayley hyperbolic plane, the Dehn function is at most
quadratic \cite{GroAII}.  If $X$ is the complex or quaternionic
hyperbolic plane, the Dehn function is cubic \cite{GroAII,PittetCLB};
if $X$ is the Cayley hyperbolic plane, the precise growth rate is
unknown, but is at most cubic.

When $G$ has $\R$-rank 2 and $\Gamma$ has $\Q$-rank 1 or 2, Leuzinger
and Pittet \cite{LeuzPitRk2} proved that the Dehn function grows
exponentially.  As in the $\R$-rank 1 case, the proof relies
on understanding the subgroups corresponding to the removed horoballs,
but in this case the subgroups are solvable and have exponential Dehn
function.  Finally, when $G$ has $\R$-rank 3 or greater and $\Gamma$
has $\Q$-rank 1, Drutu \cite{DrutuFilling} has shown that the boundary
of a horoball satisfies a quadratic filling inequality and that
$\Gamma$ enjoys an ``asymptotically quadratic'' Dehn function, i.e.,
its Dehn function is bounded by $n^{2+\epsilon}$ for any $\epsilon>0$.

When $\Gamma$ has $\Q$-rank larger than $1$, the geometry of the
space becomes more complicated.  The main difference is that
the removed horoballs are no longer disjoint, so many of the previous
arguments fail.  In many cases, the best known result is due to
Gromov, who sketched a proof that the Dehn function of $\Gamma$ is
bounded above by an exponential function \cite[5.A$_7$]{GroAII}.  A
full proof of this fact was given by Leuzinger \cite{LeuzingerPolyRet}.  

In this paper, we consider $\SL(n,\Z)$.  This is a lattice with
$\Q$-rank $n-1$ in a group with $\R$-rank $n-1$, so when $n$ is small,
the methods above apply.  When $n=2$, the group $\SL(2,\Z)$ is
virtually free, and thus hyperbolic.  As a consequence, its Dehn
function is linear.  When $n=3$, the result of Leuzinger and Pittet
mentioned above implies that the Dehn function of $\SL(3,\Z)$ grows
exponentially; this was first proved by Epstein and Thurston
\cite{ECHLPT}.

Much less is known about the Dehn function for lattices in $\SL(n,\Z)$
when $n\ge 4$.  By the results of Gromov and Leuzinger above, the Dehn
function of any such lattice is bounded by an exponential function,
but the Dehn function may be polynomial in many cases.  Thurston
\cite{ECHLPT} conjectured that
\begin{conj}
  When $n\ge 4$, $\SL(n,\Z)$ satisfies the isoperimetric inequality
  $$\delta_{\SL(n,\Z)}(\ell)\lesssim \ell^2.$$
\end{conj}
In this paper, we will prove that
\begin{thm}\label{thm:mainthm} When $n\ge 5$, $\SL(n,\Z)$ satisfies the isoperimetric
  inequality
  $$\delta_{\SL(n,\Z)}(\ell)\lesssim \ell^4.$$
\end{thm}

In Section~\ref{sec:prelims}, we present some preliminaries, and in
Section~\ref{sec:overview}, we sketch an overview of the proof.  In
Sections~\ref{sec:fundset}--\ref{sec:filling}, we prove
Theorem~\ref{thm:mainthm}.

Some of the ideas in this work were inspired by discussions at the
American Institute of Mathematics workshop, ``The Isoperimetric
Inequality for $\SL(n,\Z)$,'' and the author would like to thank the
organizers, Nathan Broaddus, Tim Riley, and Kevin Wortman; and
participants, especially Mladen Bestvina, Alex Eskin, Martin Kassabov,
and Christophe Pittet.  The author would also like to thank Tim Riley
and Yves de Cornulier for many helpful conversations while the author
was visiting Bristol University and Universit\'e de Rennes.

\section{Preliminaries} \label{sec:prelims}

In this section, we recall several facts about $\SL(p,\Z)$,
$\SL(p,\R)$, and about Dehn functions.

We provide only a minimal introduction to Dehn functions here; for a
survey with examples, see for instance \cite{Bridson}.  The Dehn
function is a group invariant which gives one way to describe the
difficulty of determining whether a word in a group represents the
identity.  It can be described both combinatorially and geometrically,
and the interaction between these two viewpoints is often crucial.  We
first give some terminology.  If $X$ is a set, and $x_i\in X$ for
$1\le i\le n$, we call the formal product $x_1\dots x_n$ a word in
$X$.  Let $X^*$ to be the set of words in $X\cup X^{-1}$, where
$X^{-1}$ is the set of formal inverses of elements of $X$.  We denote
the empty word by $\emptyword$.  If $w\in X^*$, we can write
$w=x_1x_2\dots x_n$, and we define the length $\ell(w)$ of $w$ to be
$n$.  Note especially that these words are not reduced; that is, $x$
may appear next to $x^{-1}$.  If $X\subset H$ for some group $H$,
there is a natural evaluation map $X^*\to H$, and we say that words
represent elements of $H$.

Using these concepts, we can describe the combinatorial Dehn function.
If
$$H=\langle h_1,\dots,h_d \mid r_1,\dots,r_s\rangle$$ 
is a finitely presented group, we can let $\Sigma=\{h_1,\dots,h_d\}$
and consider words in $\Sigma^*$.  If a word $w$ represents the
identity, then there is a way to prove this using the relations.  That
is, there is a sequence of steps which reduces $w$ to the empty word,
where each step is a free expansion (insertion of a subword $x_i^{\pm
  1}x_i^{\mp1}$), free reduction (deletion of a subword $x_i^{\pm
  1}x_i^{\mp1}$), or the application of a relator (insertion or
deletion of one of the $r_i$).  We call the number of applications of
relators in a sequence its {\em cost}, and we call the minimum cost of
a sequence which starts at $w$ and ending at $\varepsilon$ the {\em
  filling area} of $w$, denoted by $\delta_H(w)$.  We then define the
{\em Dehn function} of $H$ to be
$$\delta_H(n)=\max_{\ell(w)\le n} \delta_H(w),$$
where the maximum is taken over words representing the identity.  This
depends {\em a priori} on the chosen presentation of $H$; we will see
that the growth rate of $\delta_H$ is independent of this choice.  For
convenience, if $v,w$ are two words representing the same element of
$H$, we define $\delta_H(v,w)=\delta_H(vw^{-1})$; this denotes the
minimum cost to transform $v$ to $w$.

This can also be interpreted geometrically.  If $K_H$ is the {\em
  presentation complex} of $H$ (a simply-connected 2-complex whose
1-skeleton is the Cayley graph of $H$ and whose $2$-cells correspond
to translates of the relators), then $w$ corresponds to a closed curve
in the $1$-skeleton of $K_H$.  Similarly, the sequence of steps
reducing $w$ to the identity corresponds to a homotopy contracting
this closed curve to a point.  More generally, if $X$ is a riemannian
manifold or simplicial complex, we can define the filling area
$\delta_X(\gamma)$ of a Lipschitz curve $\gamma:S^1\to X$ to be the
infimal area of a Lipschitz map $D^2\to X$ which extends $\gamma$.
Then we can define the Dehn function of $X$ to be
$$\delta_X(n)=\sup_{\ell(\gamma)\le n} \delta_X(\gamma),$$
where the supremum is taken over null-homotopic closed curves.
As in the combinatorial case, if $\beta$ and $\gamma$ are two
curves connecting the same points which are homotopic with their
endpoints fixed, we define $\delta_X(\beta,\gamma)$ to be the infimal
area of a homotopy between $\beta$ and $\gamma$ which fixes their
endpoints.

Gromov stated a theorem connecting these two definitions, proofs of
which can be found in \cite{Bridson} and \cite{BurTab}:
\begin{thm}[Gromov's Filling Theorem]\label{thm:GroFill}
  If $X$ is a simply connected riemannian manifold or simplicial
  complex and $H$ is a finitely presented group acting properly
  discontinuously, cocompactly, and by isometries on $M$, then
  $\delta_H\sim \delta_M$.
\end{thm}
Here, $\sim$ is an equivalence relation which requires that $\delta_H$
and $\delta_M$ have the same growth rate according to the following
definition: if $f,g:\N\to \N$, let $f\lesssim g$ if and only if there
is a $c$ such that
$$f(n)\le c g(cn+c)+c\text{ for all }n$$
and $f\sim g$ if and only if $f\lesssim g$ and $g\lesssim f$.  One
consequence of Theorem~\ref{thm:GroFill} is that the Dehn functions
corresponding to different presentations are equivalent under this
relation.  

We state the following lemma, which is used in the proof of
Theorem~\ref{thm:GroFill}.  The lemma follows from the Federer-Fleming
Deformation Lemma \cite{FedFlem} or from the Cellulation Lemma
\cite[5.2.3]{Bridson}:
\begin{lemma} \label{lem:approx} Let $H$ and $X$ be as in the Filling
  Theorem, and let $f:K_H\to X$ be an $H$-equivariant map of a
  presentation complex for $H$ to $X$.  There is a $c$ such that:
  \begin{enumerate}
  \item Let $s:[0,1]\to X$ connect $f(e)$ and $f(h)$, where $e$ is the
    identity in $H$ and $h\in H$.  There is a word $w$ which
    represents $h$ and which has length $\ell(w)\le c\ell(s)+c$.  If
    $X$ is simply connected, then $w$ approximates $s$ in the sense
    that if $\gamma_w:[0,1]\to K_H$ is the curve corresponding to $w$,
    then
    $$\delta_X(s,\gamma_w)\le c \ell(s)+c.$$
  \item If $w$ is a word representing the identity in $H$ and
    $\gamma:S_1\to K_H$ is the corresponding closed curve in $K_H$,
    then
    $$\delta_H(w)\le c(\ell(w)+\delta_X(f\circ \gamma)).$$
  \end{enumerate}
\end{lemma}

We now set out notation for $\SL(p)$ and several of its subgroups.  In
the following, $\K$ represents either $\Z$ or $\R$; when it is
omitted, we take it to be $\R$.  Let $G=\SL(p,\R)$ and let
$\Gamma=\SL(p,\Z)$.  Let $z_1,\dots,z_p$ generate $\Z^p$, and if
$S\subset \{1,\dots,p\}$, let $\K^S=\langle z_s\rangle_{s\in S}$ be a
subspace of $\K^p$.  If $q\le p$, there are many ways to include
$\SL(q)$ in $\SL(p)$.  Let $\SL(S)$ be the copy of $\SL(\#S)$ in
$\SL(p)$ which acts on $\R^S$ and fixes $z_t$ for $t\not \in S$.  If
$S_1,\dots,S_n$ are disjoint subsets of $\{1,\dots, p\}$ such that
$\bigcup S_i=\{1,\dots,p\}$, let
$$U(S_1,\dots,S_n;\K)\subset \SL(p,\K)$$
be the subgroup of matrices preserving the flag
$$\R^{S_i}\subset \R^{S_i\cup S_{i-1}} \subset \dots\subset \R^p $$
when acting on the right.  If the $S_i$ are sets of consecutive
integers in increasing order, $U(S_1,\dots,S_n;\K)$ is block upper
triangular.  For example, $U(\{1\},\{2,3,4\};\K)$ is the subgroup of
$\SL(4,\K)$ consisting of matrices of the form:
$$\begin{pmatrix}
  * & * & * & * \\
  0 & * & * & * \\
  0 & * & * & * \\
  0 & * & * & *
\end{pmatrix}.$$ If $d_1,\dots d_n>0$, let $U(d_1,\dots,d_n;\K)$ be
the group of upper block triangular matrices with blocks of the given
lengths, so that the subgroup illustrated above is $U(1,3;\K)$.  Each
group $U(d_1,\dots,d_n;\Z)$ is a parabolic subgroup of $\Gamma$,
and any parabolic subgroup of $\Gamma$ is conjugate to a unique such
group.  Let $\cP$ be the set of these groups.

We will note some facts about the combinatorial group theory of
$\Gamma$ and its subgroups.  Let $I$ be the identity matrix.  If $1\le
i\ne j\le p$, let $e_{ij}(x)\in \SL(p,\Z)$ be the elementary matrix
which consists of the identity matrix with the $(i,j)$-entry replaced
by $x$.  Let $e_{ij}:=e_{ij}(1)$.  There is a finite presentation which has the matrices
$e_{ij}$ as generators \cite{Milnor}:
\begin{align}
  \notag  \SL(p,\Z)=\langle e_{ij} \mid \; &[e_{ij},e_{kl}]=I & \text{if $i\ne l$ and $j\ne k$}\\
  & [e_{ij},e_{jk}]=e_{ik} & \text{if $i\ne k$}\label{eq:steinberg}\\
  \notag  & (e_{ij} e_{ji}^{-1} e_{ij})^4=I \rangle,
\end{align}
where we adopt the convention that $[x,y]=xyx^{-1}y^{-1}$.  

We will use a slightly expanded set of generators.  Let
$$\Sigma=\Sigma(p)=\{e_{ij}\mid 1\le i\ne j\le p\}\cup D,$$
where $D$ is the set of diagonal matrices in $\SL(p,\Z)$.  Then there
is a finite presentation of $\SL(p,\Z)$ with generating set $\Sigma$
and relations consisting of those in \eqref{eq:steinberg} and
relations expressing each element of $D$ as a product of elementary
matrices.  The advantage of this generating set is that if
$H=\SL(S,\Z)$ or $H=U(S_1,\dots,S_n;\Z)$, then $H$ is generated by
$\Sigma \cap H$.

The group $\Gamma$ is a lattice in $G=\SL(p,\R)$, and the geometry of
$G$ and of the quotient will be important in our proof.  We think of
$G$ and $\Gamma$ as acting on the symmetric space on the left.  Let
$\cE=\SL(p,\R)/\SO(p,\R)$.  The tangent space of $\cE$ at the
identity, $T_{I}\cE$ is isomorphic to the space of symmetric matrices
with trace 0.  If $u^{tr}$ represents the transpose of $u$, then we
can define an inner product $\langle u,v\rangle=\trace(u^{tr}v)$ on
$T_I \cE$.  Since this is $\SO(p)$-invariant, it gives rise to a
$G$-invariant riemannian metric on $\cE$.  Under this metric, $\cE$ is
a non-positively curved symmetric space.  The lattice $\Gamma$ acts on
$\cE$ with finite covolume, but the action is not cocompact.  Let
$\cM:=\Gamma\backslash \cE$.  If $x\in G$, we write the equivalence
class of $x$ in $\cE$ as $[x]_\cE$; similarly, if $x\in G$ or $x\in
\cE$, we write the equivalence class of $x$ in $\cM$ as $[x]_\cM$.

If $g\in G$ is a matrix with coefficients $\{g_{ij}\}$, we define
$$\|g\|_2=\sqrt{\sum_{i,j}g_{ij}^2},$$
$$\|g\|_\infty=\max_{i,j}|g_{ij}|.$$
Note that for all $g,h\in G$, we have 
$$\|gh\|_2\le \|g\|_2\|h\|_2$$
$$\|g^{-1}\|_2\ge \|g\|^{1/p}_2$$
and that there is a $c$ such that
$$c^{-1} d_G(I,g)-c\le  \log \|g\|_2 \le c d_G(I,g)+c.$$

It will be useful to have a geometric picture of elements of $\cE$ and
$\cM$.  The rows of a matrix in $\SL(p,\R)$ give a unit volume basis
of $\R^p$, and we can think of $G$ as the set of such bases.  From
this viewpoint, $\SO(p)$ acts on a basis by rotating the basis
vectors, so $\cE$ consists of the set of bases up to rotation.  An
element of $\Gamma$ acts by replacing the basis elements by integer
combinations of basis elements.  This preserves the lattice that they
generate, so we can think of $\Gamma\backslash G$ as the set of
unit-covolume lattices in $\R^p$.  The quotient $\cM$ is then the set
of unit-covolume lattices up to rotation.  Nearby points in $\cM$ or
$\cE$ correspond to bases or lattices which can be taken into each
other by small linear deformations of $\R^p$.

Finally, we define a subset of $\cE$ on which $\Gamma$ acts
cocompactly.  Let $\cE(\epsilon)$ be the set of points which
correspond to lattices with injectivity radius at least $\epsilon$.
When $\epsilon \le 1/2$, this set is contractible and $\Gamma$ acts on
it cocompactly \cite{ECHLPT}; we call it the {\em thick part} of
$\cE$, and its preimage $G(\epsilon)$ in $G$ the thick part of $G$.  Let
$\iota:K_\Gamma\to \cE$ be a $\Gamma$-equivariant map; if $\epsilon$
is sufficiently small, then the image of $\iota$ is contained in
$\cE(\epsilon)$.

\section{Overview of proof} \label{sec:overview} To understand our
methods for proving a polynomial Dehn function for $\Gamma$, it is
helpful to consider a related method for proving an exponential Dehn
function.  Let $w\in \Sigma^*$ be a word which represents the identity
in $\Gamma$, so that $w$ corresponds to a closed curve in $K_\Gamma$.
By abuse of notation, we also call this curve $w$.  We can construct a
curve $\alpha:S^1\to \cE$ which corresponds to $w$ by letting
$\alpha=[\iota(w)]_\cE$.  Let $\ell=\ell(\alpha)$ and assume
that $\alpha$ is parameterized by length.

Since $\cE$ is non-positively curved, we can use geodesics to fill
$\alpha$.  If $x,y\in \cE$, let $\lambda_{x,y}:[0,1]\to \cE$ be a
geodesic parameterized so that $\lambda_{x,y}(0)=x$,
$\lambda_{x,y}(1)=y$, and $\lambda_{x,y}$ has constant speed.  We can
define a homotopy $h:[0,\ell] \times [0,1]\to \cE$ by
$$h(x,t)=\lambda_{\alpha(x),\alpha(0)}(t/\ell).$$
Let $D^2\subset \R^2$ be the disc of radius $\ell$ centered at the
origin and let
$$f(r,\theta)=h(\ell \frac{\theta}{2\pi},r/\ell)$$
where $r$ and $\theta$ are polar coordinates.  Since $\cE$ is
non-positively curved, this map is Lipschitz and its Lipschitz
constant $\Lip(f)$ is bounded independently of $\alpha$; in
particular, it has area $O(\ell^2)$.  Furthermore, the image of $f$ is
contained in a ball around $[I]_\cE$ of radius $\ell$.

Since $\Gamma$ does not act cocompactly on $\cE$, this filling does
not directly correspond to an efficient filling of $w$ in $K_\Gamma$.
To construct a filling in $K_\Gamma$, we will need a map $\rho:\cE\to
\Gamma$.  We can construct one from a fundamental set for the action
of $\Gamma$ on $\cE$; we let $\cS$ be a Siegel set (see Sec.\
\ref{sec:fundset}) and define $\rho$ so that for all $x\in \cE$, $x\in
\rho(x)\cS$.  Since $\cM$ is not compact, this map is not a
quasi-isometry; if $x\in \cE$ is deep in the cusp of $\cM$, then small
changes in $x$ can result in large changes in $\rho(x)$.  On the other
hand, the injectivity radius in the cusp shrinks at most exponentially
with the distance from a basepoint.  That is, there is a $c$ such that
if $x\in B_r(I)\subset \cE$, and $d_{\cE}(x,y)<\exp(-c r),$ then
$d_{\Gamma}(\rho(x),\rho(y))\le c$.

Our basic technique is to construct a triangulation $\tau$ of the
disc, and use $f$ as a template for a map $\bar{f}:\tau\to K_\Gamma$.
We will construct $\bar{f}:\tau\to K_\Gamma$ one dimension at a time.
Let $\tau$ be a triangulation of $D^2$ with
$O(e^{2c\ell})$ cells such that the image of each cell under $f$ has
diameter at most $e^{-c\ell}$.  If $x$ and $y$ are vertices of an edge
of $\tau$, then
$$d_{\Gamma}(\rho(f(x)),\rho(f(y)))\le c,$$
where $d_{\Gamma}$ is the word metric on $\Gamma$ given by the
generating set $\Sigma$.  Let $\bar{f}_0:\tau^{(0)}\to
K_\Gamma$ be given by $\bar{f}_0(x)=\rho(f(x))$ for all $x$ (where we
identify elements of $\Gamma$ with the corresponding vertices in
$K_\Gamma$).

To construct $\bar{f}_1:\tau^{(1)}\to K_\Gamma$, we must find words in
$\Sigma^*$ which connect the images of adjacent vertices of $\tau$;
that is, for each edge $e=(x,y)$, we must find a word in $\Sigma^*$
representing $\bar{f}_0(x)^{-1}\bar{f}_0(y)$.  Since $\bar{f}_0(x)^{-1}\bar{f}_0(y)$ is a bounded
element of $\Gamma$, we choose $\bar{f}_1(e)$ to be a
word of length at most $c$.

Finally, we construct $\bar{f}$ on the triangles of $\tau$.  If
$\Delta$ is a triangle of $\tau$, then $\bar{f}_1(\partial\Delta)$
corresponds to a word of length at most $3c$ which represents the
identity.  Since $K_\Gamma$ is simply connected, each such word can be
filled by a disc of area at most $\delta_{\Gamma}(3c)$.  This
results in a map $\bar{f}$ of area $O(e^{2c\ell})$.  The
boundary of $\bar{f}$ is not quite $w$, but it remains a bounded
distance from $w$, and there is a homotopy between the two of area
$O(\ell)$.  Thus $\delta_\Gamma(w)=O(e^{2c\ell})$, and
$$\delta_{\Gamma}(\ell)\lesssim e^{\ell},$$ as desired.

We will prove a polynomial bound with a similar scheme.  The main
difference is that we construct $\tau$ by dividing $D^2$ into
$O(\ell^2)$ triangles of diameter $\le 1$ instead of exponentially
many triangles of exponentially small diameter.  We define $\rho$ and
$\bar{f}_0$ as described above, but it is no longer the case that if
$x$ and $y$ are connected by an edge, then
$d_{\Gamma}(\bar{f}_0(x),\bar{f}_0(y))<c$.  In Section~\ref{sec:parabounds}, we use the geometry of
$\cM$ to show instead that $\bar{f}_0(x)^{-1}\bar{f}_0(y)$ is the
product of a block-diagonal element of $\Gamma$ with bounded
coefficients and a unipotent element with at most exponentially large
coefficients. 

Because $\bar{f}_0(x)^{-1}\bar{f}_0(y)$ is no longer a bounded element
of $\Gamma$, we must change the way we define $\bar{f}_1$ as well.  In
Section~\ref{sec:normalform}, we will define a normal form for block
upper-triangular matrices with bounded block-diagonal part and
exponentially large unipotent part.  We will replace edges of
$\tau$ with words in this normal form which have length $O(\ell)$.

Finally, we construct $\bar{f}$ by extending $\bar{f}_1$ to the
2-cells of $\tau$.  The boundary of each 2-cell is a product of three
words in normal form and has length $O(\ell)$; in
Section~\ref{sec:filling}, we will show that such words can be filled
with discs of area $O(\ell^2)$.  Since there are $O(\ell^2)$ such
triangles to fill, this method will give an $\ell^4$ upper bound on
the Dehn function.

\section{Constructing a fundamental set}\label{sec:fundset}
In this section, we will define $\cS$, a fundamental set for $\Gamma$.
Let $\diagmat(t_1,\dots, t_p)$ be the diagonal matrix with entries $(t_1,\dots, t_p)$.
Let $A$ be the set of diagonal matrices in $G$ and let
$$A^+_{\epsilon}=\{\diagmat(t_1,\dots, t_p)\mid \prod t_i=1, t_i > 0, t_i\ge \epsilon t_{i+1}\}.$$
Let $N$ be the set of upper triangular matrices with 1's on the
diagonal and let $N^+$ be the subset of $N$ with off-diagonal entries
in the interval $[-1/2,1/2]$.  Translates of the set $N^+A^+_\epsilon$
are known as Siegel sets.  The following properties of Siegel sets are
well known (see for instance \cite{BorHar-Cha}).
\begin{lemma}\label{lem:redThe}\ \\
  There is an $1>\epsilon_{\cS}>0$ such that if we let 
  $$\cS:=[N^+A^+_{\epsilon_{\cS}}]_\cE\subset \cE,$$
  then
  \begin{itemize}
  \item $\Gamma\cS=\cE$. \label{lem:redThe:cover}
  \item There are only finitely many elements $\gamma\in \Gamma$
    such that $\gamma \cS \cap \cS \ne \emptyset$.  \label{lem:redThe:fundSet}
  \end{itemize}
\end{lemma}
We define $A^+:=A^+_{\epsilon_{\cS}}$.  Translates of $\cS$ cover all
of $\cE$, so we can define a map $\rho:\cE\to \Gamma$ such that
$\rho(\cS)=I$ and $x\in \rho(x)\cS$ for all $x$.  As in
Section~\ref{sec:overview}, we define $\bar{f}_0:\tau^{(0)}\to
K_\Gamma$ by $\bar{f}_0(x)=\rho(f(x))$.

The inclusion $A^+ \hookrightarrow \cS$ is a Hausdorff equivalence:
\begin{lemma}\label{lem:easyHausdorff}
  Give $A$ the riemannian metric inherited from its inclusion in $G$, so that
  $$d_{A}(\diagmat(d_1,\dots,d_p),\diagmat(d'_1,\dots,d'_p))=\sqrt{\sum_{i=1}^p \left|\log \frac{d'_i}{d_i}\right|^2}.$$
  \begin{itemize}
  \item There is a $c$ such that if $x\in \cS$, then
    $d_\cE(x,[A^+]_\cE)\le c$.
  \item If $x,y\in A^+$, then $d_{A}(x,y)=d_{\cS}(x,y)$.
  \end{itemize}
\end{lemma}
\begin{proof}
  For the first claim, note that if $x=[na]_\cE$, then $x=[a(a^{-1} n
  a)]_\cE$, and $a^{-1}na\in N$.  Furthermore,
  $$\|a^{-1}na\|_\infty\le \epsilon_\cS^p,$$
  so 
  $$d_\cE([x]_\cE,[a]_\cE)\le d_G(I,a^{-1}na)$$
  is bounded independently of $x$.

  For the second claim, we clearly have $d_{A}(x,y)\ge
  d_{\cS}(x,y)$.  For the reverse inequality, it suffices to note that
  the map $\cS\to A^+$ given by $na\mapsto a$ for all $n\in N^+$,
  $a\in A^+$ is distance-decreasing.
\end{proof}
Siegel conjectured that the quotient map from $\cS$ to $\cM$ is also a
Hausdorff equivalence, that is:
\begin{thm}\label{thm:SiegConj}
  There is a $c$ such that if $x,y\in \cS$, then 
  $$d_{\cS}(x,y)-c\le d_{\cM}([x]_\cM,[y]_\cM)\le d_{\cS}(x,y)$$
\end{thm}
Proofs of this conjecture can be found in \cite{Leuzinger,Ji,Ding}.
One consequence is that $A^+$ is Hausdorff equivalent to $\cM$, and it
will be helpful to have a map $\phi_\cM:\cM\to A^+$ which realizes
this Hausdorff equivalence.  Ji and MacPherson \cite{JiMacPherson}
used precise reduction theory to define such a map in a more general
setting.  In the special case that $G=\SL(n)$ and $\Gamma=\SL(n,\R)$,
their map and the map $\phi_\cM$ that we will define differ by a
bounded distance.

Any point $x\in \cE$ can be written as $x=[\gamma na]_\cE$ for some
$\gamma\in \Gamma$, $n\in N^+$ and $a\in A^+$ in at most finitely many
ways.  These decompositions have the following property:
\begin{cor}[see \cite{JiMacPherson}, Lemmas 5.13, 5.14]\label{cor:Hausdorff}
  There is a constant $c_{\phi}$ such that if $x,y\in \cE$,
  $\gamma,\gamma' \in \Gamma$, $n,n' \in N^+$ and $a,a'\in A^+$ are
  such that $x=[\gamma na]_\cE$ and $y=[\gamma' n'a']_\cE$, then
  $$|d_{\cM}([x]_\cM,[y]_\cM)- d_{A}(a,a')|\le c_\phi.$$ 

  In particular, if $[\gamma na]_\cE=[\gamma' n'a']_\cE$, then 
  $$d_{A}(a,a')\le c_\phi.$$
\end{cor}
\begin{proof}
  Without loss of generality, we may assume that $\gamma=\gamma'=I$.
  Let $c$ be as in Theorem~\ref{thm:SiegConj} and let $c'$ be as in
  Lemma~\ref{lem:easyHausdorff}, so that
  $$d_{\cM}([na]_\cM,[a]_\cM)\le d_{\cE}([na]_\cE,[a]_\cE)\le c'.$$
  Then
  \begin{align*}
    |d_{\cM}([x]_\cM,[y]_\cM)- d_{A}(a,a')|&=|d_{\cM}([na]_\cM,[n'a']_\cM)- d_{A}(a,a')| \\
    &\le c+ |d_{\cS}([na]_\cE,[n'a']_\cE)- d_{\cS}([a]_\cE,[a']_\cE)| \\
    &\le c+ 2c'.
  \end{align*}
\end{proof}

We now define $\phi_\cM$.  
Any point $x\in \cE$ can be uniquely
written as $x=[\rho(x)na]_\cE$ for some $n\in N^+$ and $a\in A^+$.
Let $\phi:\cE\to A^+$ be the map $[\rho(x)na]_\cE\mapsto a$.  This is
not quite $\Gamma$-equivariant, but we can still define a map
$\phi_\cM:\cM\to A^+$ by choosing a lift $\tilde{x}\in \cE$ for all
$x\in \cM$ and defining $\phi_{\cM}(x)=\phi(\tilde{x})$.  By the
corollary, $\phi_{\cM}(x)$ is a Hausdorff equivalence with constant
$c_\phi$.

\section{Bounding group elements corresponding to edges}\label{sec:parabounds}
In this section, we will restrict the possible values of
$\rho(x)^{-1}\rho(y)$ when $x,y\in \cE$ and $d_{\cE}(x,y)\le 1$.  This
is a key step in extending $\bar{f}_0$ to the
$1$-skeleton of $\tau$.  

The possible values of $\rho(x)^{-1}\rho(y)$ depend on $\phi(x)$.  We
will construct a cover of $A^+$ by sets corresponding to parabolic
subgroups so that the possible values of
$\bar{f}_0(x)^{-1}\bar{f}_0(y)$ depend on which set $\phi(x)$ falls into.
If $P=U(d_1,\dots,d_r)$, where $\sum d_i=p$, let $s_i=\sum_{j=1}^i d_i$
for $0\le i\le r$.  Let
$$X_P(t)=\{\diagmat(a_1,\dots,a_p)\in A\mid t a_{i+1}<a_i \text{ if and only if }  i\in\{s_1,\dots,s_{r-1}\}\}.$$
These sets partition $A^+$ into $2^{p-1}$ disjoint subsets.

If $\phi(x)\in X_P(t)$ for some sufficiently large $t$, then the
geometry of the lattice corresponding to $x$ is quite distinctive.
Recall that if $\tilde{x}\in G$ is a representative of $x\in \cE$,
then we can construct a lattice $\Z^p \tilde{x}\subset \R^p$, and
different representatives of $x$ correspond to rotations of $\Z^p
\tilde{x}$.  Let
$$V(x,r)=\langle v\in \Z^p\mid \|v \tilde{x}\|_2 \le r  \rangle;$$
this corresponds to the subspace of the lattice generated by vectors
of length at most $r$, and is independent of the choice of
$\tilde{x}$.  As such, $V(x,r)$ is $\Gamma$-equivariant: if $\gamma\in
\Gamma$, then $V(\gamma x,r)=V(x,r)\gamma^{-1}$.  In many cases,
$\phi(x)$ and $\rho(x)$ determine $V(x,r)$.  Let
$z_1,\dots,z_p\in \Z^p$ be the standard generating set of $\Z^p$, and
let $Z_j=\langle z_{j},\dots, z_{p}\rangle$.
\begin{lemma}\label{lem:phiFlag}
  There is a $c_V>1$ such that if $x\in \cE$,
  $\phi(x)=\diagmat(a_1,\dots,a_p)$, and
  $$a_{j+1}c_V<r<c_V^{-1}a_{j},$$
  then $V(x,r)=Z_j\rho(x)^{-1}$.
\end{lemma}
\begin{proof}
  It suffices to show that if the hypotheses hold and $x\in \cS$, then
  $V(x,r)=Z_j$.  There is an $n=\{n_{ij}\}\in N^+$ such that
  $x=[n\phi(x)]_\cE$, and if $\tilde{x}=n\phi(x)$, then
  \begin{align*}
    z_j \tilde{x} &= z_j n\phi(x) \\
    &= a_j z_j+\sum_{i=j+1}^p n_{ji} z_ia_{i}.
  \end{align*}
  Since
  $|n_{ji}|\le 1/2$ when $i>j$ and $a_{i+1}\le a_i \epsilon_{\cS}^{-1}$,
  we have
  $$\|z_j \tilde{x}\|_2 \le a_j\sqrt{p}\epsilon_{\cS}^{-p},$$
  so 
  $$V(x,a_j\sqrt{p}\epsilon_{\cS}^{-p})\supset Z_j.$$  On the
  other hand, if $v\not \in Z_j$, then $v=\sum_i v_i z_i$ for some
  $v_i\in \Z$.  Let $k$ be the smallest $k$ such that $v_k\ne 0$; by
  assumption, $k<j$.  The $z_{k}$-coordinate of $v\tilde{x}$ is
  $v_{k} a_{k}$, so
  $$\|v n\phi(y)\|_2 \ge |a_{k}|> a_{j-1} \epsilon_{\cS}^{p}$$
  and thus if $t<a_{j-1} \epsilon_{\cS}^{p}$, then $V(x,t)\subset Z_j$.
  Therefore, if 
  $$a_j\sqrt{p}\epsilon_{\cS}^{-p}\le t<a_{j-1} \epsilon_{\cS}^{p},$$
  then $V(\tilde{x},t)=Z_j$.  
\end{proof}
In particular, if $\phi(x)\in X_P(2 c_V^2)$, then
$a_{s_i+1}c_V<c_V^{-1}a_{s_i}$ and we can find $r_i$ such
that $V(x,r_i)=Z_{s_i} \rho(x)^{-1}$.  

Let $M_P$ be the subgroup of $P$ consisting of block diagonal
matrices, so that $M_P$ contains
$\SL(d_1)\times \dots \times \SL(d_n)$
as a finite index subgroup.
Let $N_P\subset P$ be the subgroup of block upper triangular matrices
whose diagonal blocks are the identity matrix.  Any element $z\in P$
can be uniquely decomposed as a product $z=nm$, where $n\in N_P$ and
$m\in M_P$; we call $m$ the $P$-reductive part of $z$ and $n$ the
$P$-unipotent part.  We will show that if $d(x,y)\le 1$, then
$z=\rho(x)^{-1}\rho(y)\in P$ for some $P$, where the $P$-reductive
part of $z$ has coefficients bounded independently of $\ell$ and the
$P$-unipotent part has coefficients at most exponential in $\ell$.
\begin{lemma} \label{lem:paraBounds} Let $p\ge 3$.  There is a $t_0>0$
  and a $c_\rho>0$ such that for all $P\in \cP$ and all $x,y\in \cE$ such
  that $\phi(x)\in X_P(t_0)$ and $d_{\cE}(x,y)\le 1$,
  we can decompose $\rho(x)^{-1}\rho(y)$ as a product
  $\rho(x)^{-1}\rho(y)=nm$, where $n\in N_P(\Z)$, $m\in M_P(\Z)$,
  $d_{\Gamma}(I,m)<c_\rho$, and $\|n\|_2 \le c_\rho e^{c_\rho d_{\cE}(x,[I]_\cE)}$.
\end{lemma}
\begin{proof}
  Let 
  $$t_0=2\exp(4(c_\phi+1))c_V^2,$$
  let $P=U(d_1,\dots,d_n)$, and let $x$ and $y$ be as in the
  hypothesis of the lemma.  By translating $x$ and $y$ by
  $\rho(x)^{-1}$, we may assume that $x\in \cS$ and thus $\rho(x)=I$.  We first claim that
  $\rho(y)\in P(\Z)$.

  Let $a_i,a'_i$ be such that $\phi(x)=\diagmat(a_1,\dots,a_p)$ and
  let $\phi(y)=\diagmat(a'_1,\dots,a'_p)$.  Let $r_i=a_{s_i}
  \sqrt{t_0}$.  We claim that for all $i$,
  $$Z_{s_i}=V(x,r_i)=V(y,r_i)=Z_{s_i}\rho(y)^{-1}.$$

  The fact that $Z_{s_i}=V(x,r_i)$ follows from
  Lemma~\ref{lem:phiFlag}; in fact,
  $V(x,e^{-1}r_i)=V(x,er_i)=Z_{s_i}$.

  Since $d_{\cE}(x,y)\le 1$,
  the lattices corresponding to $x$ and $y$ only differ by a small
  deformation; this deformation can change the length of a vector in
  the lattice by at most a factor of $e$.  Thus, if $v\in \Z^p$, then
  $$e^{-1}\le\frac{\|v\tilde{x}\|_2}{\|v\tilde{y}\|_2}\le e,$$
  and in particular, 
  $$V(x,e^{-1} r_i)\subset V(y,r_i)\subset V(x,e r_i).$$
  Since the outer two sets are equal, we have $V(x,r_i)=V(y,r_i)$.  

  Finally, we need to show that $V(y,r_i)=Z_{s_i}\rho(y)^{-1}$.  By
  the lemma, it suffices to show that
  $a'_{s_i}c_V<r_i<c_V^{-1}a'_{s_i}.$  By
  Corollary~\ref{cor:Hausdorff}, we know that
  $d_{A}(\phi(x),\phi(y))\le c_\phi+1$; in particular,
  $$\left|\log\frac{a_i}{a'_{i}}\right|\le  c_\phi+1,$$
  and $a'_{s_i}c_V<r_i<c_V^{-1}a'_{s_i}$ as desired.
  Thus $Z_{s_i}=Z_{s_i}\rho(y)^{-1}$ for all $i$, so $\rho(y)\in P$.  

  We decompose $\rho(y)$ as a product $\rho(y)=n_ym_y$,
  where $n_y\in N_P(\Z)$ and $m_y\in M_P$ consists of the diagonal blocks of $\rho(y)$.  We will
  bound $m_y$ by constructing a map from $\cE$ to a product of
  symmetric spaces.

  Let $A_P\subset P$ be the subgroup consisting of diagonal matrices
  whose diagonal blocks are scalar matrices with positive
  coefficients; this is isomorphic to $(\R^+)^{n-1}$.  The parabolic
  subgroup $P$ can be uniquely decomposed according to the Langlands
  decomposition as $P=N_PM_PA_P$, and we can define a map $\mu:P\to
  M_P$ so that if $g=n m a$, where $n\in N_P$, $m\in M_P$, and $a\in
  A_P$, then $\mu(g)=m$.  Furthermore, since $M_P$ normalizes $A_P$
  and $N_P$, this is a homomorphism.

  This descends to a map on symmetric spaces; if we let
  $K_P=\SO(p)\cap P$, we get a map $\mu_{\cE}:\cE \to M_P/K_P$.  This
  map is Lipschitz.  Furthermore, if $p\in P$, $x\in \cE$, then
  $\mu_{\cE}(px)=\mu(p)\mu_{\cE}(x)$.

  This map can be interpreted geometrically.  Note that $M_P/K_P$ is a
  product of symmetric spaces of lower dimensions, so $\mu_{\cE}$
  breaks a lattice in $\R^p$ into lattices in lower-dimensional
  subspaces.  Let $V_0=\{0\}$ and $V_i=Z_{s_i}$, so that $P$ preserves
  the flag $V_0\subset \dots\subset V_n=\Z^p$.  Then if $g\in G$ (not
  necessarily parabolic) is a representative of $x$, then
  $V_ig/V_{i-1}g$ is a $d_i$-dimensional lattice in $(\R\otimes
  V_i)g/(\R\otimes V_{i-1})g$.  This lattice generally does not have
  unit covolume, but we can rescale and possibly reflect it to a
  unit-covolume lattice.  These lattices correspond to a point in
  $$M_P/K_P=\SL(d_1,\R)/\SO(d_1)\times\dots\times \SL(d_n,\R)/\SO(d_n),$$
  and this point is $\mu_{\cE}(x)$.  

  The group $M_P$ acts on $M_P/K_P$ on the left, but this action is
  not cocompact.  We will show that $\mu_{\cE}(x)$ and $\mu_{\cE}(y)$
  lie near an orbit of this action and use this to show that $\rho(y)$
  is bounded.  Let $B:=B_{c_\phi+1}(X_P(t_0),A^+)$ be a neighborhood of
  $X_P(t_0)$ in $A^+$, so that $\phi(y)\in
  B$.  Let
  $$\beta_P=[P N^+  B]_\cE.$$
  If $z\in \cE$, $\rho(z)\in P$, and $\phi(z)\in
  B$, then $z\in \beta_P$; in particular, $x,y\in
  \beta_P$.

  We claim that the image of $\beta_P\cap \cS$ is a bounded set in
  $M_P/K_P$.  If $b\in \beta_P\cap \cS$, there is a unique
  decomposition $b= [n_b a_b]_\cE$, where $n_b\in N^+$ and $a_b\in B$,
  and $\mu_\cE(b)=[\mu(n_b)\mu(a_b)]_\cE$.  Since $N^+$ is compact,
  $\mu(n_b)$ is bounded.  Since $a_b\in B$, the ratio of two
  coefficients in a diagonal block of $a_b$ is bounded, and so
  $\mu(a_b)$ is bounded as well.  Thus $\mu_\cE(\beta_P\cap \cS)$ is
  bounded; call this set $\omega_P$.

  Since $x\in\beta_P\cap \cS$ and $y\in \beta_P\cap
  \rho(y)\cS=\rho(y)(\beta_P\cap \cS)$, we know $\mu_\cE(x)\in
  \omega_P$ and $\mu_\cE(y)\in m_y \omega_P$.  Since $M_P(\Z)$ acts
  properly discontinuously on $M_P/K_P$ and
  $d_{M_P/K_P}(\mu_\cE(x),\mu_\cE(y))\le \Lip(\mu_\cE)$, there are
  only finitely many possibilities for $m_y$.

  To bound $n_y$, write $x$ and $y$ as $x=n\phi(x)\SO(p)$ and
  $y=\rho(y)n'\phi(y)\SO(p)$ for some $n,n'\in N^+$.  Since $d_\cE(x,y)\le
  1$, there is a $c$ such that
  $$\|(n\phi(x))^{-1}\rho(y)n'\phi(y)\|_2<c.$$
  and thus
  \begin{align*}
    \|\rho(y)\|_2&< \|n\phi(x)\|_2\|(n\phi(x))^{-1}\rho(y)n'\phi(y)\|_2\|(n'\phi(y))^{-1}\|_2\\
    \log \|\rho(y)\|_2&<\log c+\log \|n\phi(x)\|_2+\log \|(n'\phi(y))^{-1}\|_2\\
    &=O(d_{\cE}(I,\phi(x)) + d_{\cE}(I,\phi(y)))
  \end{align*}
  By Corollary~\ref{cor:Hausdorff}, we see that $\log
  \|\rho(y)\|_2=O(d_{\cE}(I,x))$ as desired.
\end{proof}
The work of Ji and MacPherson \cite{JiMacPherson} suggests how this
construction might be extended to lattices in other symmetric spaces.
We can replace $\phi$ with a map from the quotient to the asymptotic
cone of the quotient and replace $X_P$ with a generalized Siegel
set for $P$ and get similar results.  

In the next section, we will need the following corollary, which tells
us that if $\Delta$ is a 2-cell of $\tau$, then all the edges of $\Delta$
satisfy the conditions of Lemma~\ref{lem:paraBounds} for a single
parabolic subgroup $P$.
\begin{cor} \label{cor:triBounds} Let $x_1,x_2,x_3\in \cE$ be such
  that the distance between any pair of points is at most $1$.  There
  is a $c_\rho'$ such that if $\phi(x_1)\in X_P(t_0)$, then for all
  $i,j$, we can decompose $\rho(x_i)^{-1}\rho(x_j)$ as a product
  $\rho(x_i)^{-1}\rho(x_j)=nm$, where $n\in N_P(\Z)$, $m\in M_P(\Z)$,
  $d_{\Gamma}(I,m)<c_\rho'$, and $\|n\|_2 \le c_\rho' e^{c_\rho'
    d_{\cE}(x,I)}$.

  In particular, if $\Delta$ is a 2-cell in $\tau$, we can choose $x$
  to be a vertex of $\Delta$ and let $P_{\Delta}\in \cP$ be such that
  $\phi(f(x))\in X_{P_\Delta}(t_0)$.  Then if $y$ and $z$ are vertices
  of $\Delta$, then $\bar{f}_0(y)^{-1}\bar{f}_0(z)$ can be decomposed
  as above.
\end{cor}
\begin{proof}
  This follows from the lemma for $\rho(x_1)^{-1}\rho(x_2)$ and
  $\rho(x_1)^{-1}\rho(x_3)$, and
  $$\rho(x_3)^{-1}\rho(x_2)=(\rho(x_3)^{-1}\rho(x_1))(\rho(x_1)^{-1}\rho(x_2)).$$
\end{proof}
\section{Constructing words representing edges}\label{sec:normalform}
We will use Lemma~\ref{lem:paraBounds} to extend $\bar{f}_0$ to a map
$\bar{f}_1:\tau^{(1)}\to K_\Gamma$.  This corresponds to choosing, for
each edge $e=(x,y)$, a word $w_e$ representing
$\bar{f}_0(x)^{-1}\bar{f}_0(y)$.

If $\phi(x)\in X_P(t_0)$, we will choose a $w_e$ which is a product of
boundedly many generators of $M_P$ and boundedly many words in
$\Sigma^*$ which each represent an elementary matrix in $N_P$.  One
difficulty is doing this consistently, so that the boundary of each
triangle satisfies this condition for a single $P$.  We will need two main
lemmas.  The first states that elementary matrices with large
coefficients can be represented by ``shortcuts''.  This is a key
ingredient in the proof of the theorem of Lubotzky, Mozes, and
Raghunathan \cite{LMRComptes} which states that when $p\ge 3$, the
word metric on $\Gamma$ is equivalent to the metric induced by the
Riemannian metric on $G$; see also \cite{RileyNav} for an explicit combinatorial construction.
\begin{lemma}[see \cite{LMRComptes}] \label{lem:shortcuts}
  If $p\ge 3$, then for every $i,j\in \{1,\dots,p\}$, $i\ne j$, and
  $x\in \Z$, there is a word $\short{e}_{ij}(x)$ representing
  $e_{ij}(x)$ which has length $O(\log |x|)$.
\end{lemma}

To state the second lemma, we will need to define some sets of matrix
indices.  If $P=U(S_1,\dots,S_n)$, let
\begin{align*}
& \chi(M_P):=\{(s_1,s_2)\mid s_1,s_2\in S_i\text{ for some $i$}\},\\
& \chi(N_P):=\{(s_1,s_2)\mid s_1\in S_i,s_2\in S_j\text{ for some $i<j$}\},\\
& \chi(P):=\chi(M_P)\cup \chi(N_P)=\{(s_1,s_2)\mid s_1\in S_i,s_2\in S_j\text{ for some $i\le j$}\}.
\end{align*}

Let $P_\Delta$ be as in Corollary~\ref{cor:triBounds}.
\begin{lemma}\label{lem:edgeChar}
  If $p\ge 3$, there is a $c$ depending only on $p$ and a choice of a
  word $w_e\in \Sigma^*$ for each edge $e$ in $\tau$ such that if $\Delta$ is a 2-cell of $\tau$ and $e=(x,y)$ is an edge of $\Delta$,
  then:
  \begin{itemize}
  \item $w_e$ represents $\bar{f}_0(x)^{-1}\bar{f}_0(y)$,
  \item $\ell(w_e)=O(\ell)$,
  \item $w_e$ can be written as a product $w_e=z_1\dots z_n$ such that
    $n\le c$ and each $z_i$ is either an element of $\Sigma\cap M_{P_\Delta}$
    or a word $\short{e}_{ij}(x)$ where $(i,j)\in \chi(N_{P_\Delta})$ and
    $|x|\le c_\rho' e^{c_\rho'\ell}$, where $c_\rho'$ is the constant from Cor.~\ref{cor:triBounds}.
  \end{itemize}
\end{lemma}

\begin{proof}[Proof of Lemma~\ref{lem:shortcuts}]
  In \cite{LMRComptes}, the $\short{e}_{ij}(x)$ are constructed by
  including the solvable group $\R\ltimes \R^2$ in the thick part of
  $G$; since $\R^2\subset \R\ltimes \R^2$ is exponentially distorted,
  there are curves in $\R\ltimes \R^2$ which can be approximated by
  words in $\Gamma$.  For our purposes, we will need a construction
  which uses more general solvable groups.  In particular, when $p\ge
  4$, we can construct the $\short{e}_{ij}(x)$ as approximations of
  curves in solvable groups with quadratic Dehn function.

  Let $S,T\subset \{1,\dots,n\}$ be disjoint subsets and let $s=\#S$
  and $t=\#T$.  Assume that $s\ge 2$.  We will define a solvable
  subgroup $H_{S,T}\subset U(S,T)$.  Let $A_1,\dots, A_{s}$ be a set of
  simultaneously diagonalizable positive-definite matrices in
  $\SL(S,\Z)$.  The $A_i$'s have the same eigenvectors; call these
  shared eigenvectors $v_1,\dots, v_{s}\in \R^{S}$, and normalize them
  to have unit length.  The $A_i$ are entirely determined by their
  eigenvalues, and we can define vectors
  $$q_i=(\log \|A_iv_1\|_2,\dots,\log \|A_iv_s\|_2)\in \R^s$$
  Since $A_i\in \SL(S,\Z)$, the product of its eigenvectors is $1$,
  and the sum of the coordinates of $q_i$ is 0.  We require that the
  $A_i$ are independent in the sense that the $q_i$ span a
  $(s-1)$-dimensional subspace of $\R^s$; since they are all contained
  in an $(s-1)$-dimensional subspace, this is the maximum rank
  possible. If a set of matrices satisfies these conditions, we call
  them a set of {\em independent commuting matrices} for $S$.  A construction of such matrices can be found in Section 10.4 of
  \cite{ECHLPT}.  The $A_i$ generate a subgroup isomorphic to
  $\Z^{s-1}$, and by possibly choosing a different generating set for
  this subgroup, we can assume that $\lambda_i:=\|A_iv_i\|_2>1$ for
  all $i$.

  Let $B_1^{tr},\dots, B_{t}^{tr}\in \SL(T,\Z)$ (where $^{tr}$
  represents the transpose of a matrix) be a set of independent
  commuting matrices for $T$ and let $w_1,\dots, w_{t}\in \R^{T}$ be
  the basis of unit eigenvectors of the $B_i^{tr}$.  Choose the $B_i$
  so that $\mu_i:= \|w_i B_i\|_2 >1$.  Let
  \begin{align*}
    H_{S,T}  :=&\left\{\begin{pmatrix}\prod_i A_i^{x_i} & V \\
      0                      & \prod_i B_i^{y_i}\end{pmatrix} \middle|\; x_i,y_i\in \R, V\in \R^S\otimes \R^T\right\} \\
    =&(\R^{s-1}\times\R^{t-1})\ltimes (\R^S\otimes \R^T).
  \end{align*}
  Note that $H_{S,T}\cap \Gamma$ is a cocompact lattice in $H_{S,T}$,
  so $H_{S,T}$ is contained in the thick part of $G$.  That is, if
  $\epsilon$ is sufficiently small, then $[H_{S,T}]_\cE\subset
  \cE(\epsilon)$, so Lemma~\ref{lem:approx} can be used to construct
  words in $\Sigma^*$ out of paths in $H_{S,T}$.  We will use this and
  the fact that the subgroup $\R^S\otimes \R^T$ is exponentially
  distorted in $H_{S,T}$ to get short words in $\Sigma^*$ representing
  certain unipotent matrices.

  By abuse of notation, let $A_i$ and $B_i$ refer to the corresponding
  matrices in $H_{S,T}$.  The group $H_{S,T}$ is generated by powers
  of the $A_i$, powers of the $B_i$, and elementary matrices in the
  sense that any element of $H_{S,T}$ can be written as
  $$\prod A_i^{x_i}\prod B_i^{y_i} \begin{pmatrix} I_S & V \\ 0 & I_T \end{pmatrix},$$
  for some $x_i,y_i\in \R$ and $V\in \R^S\otimes \R^T$, where $I_S$
  and $I_T$ represent the identity matrix in $\SL(S,\Z)$ and
  $\SL(T,\Z)$ respectively.  As with discrete groups we will associate
  generators with curves, and words with concatenations of curves.  We
  let $A_i^{x}$ correspond to the curve
  $$d\mapsto \begin{pmatrix} A_i^{xd} & 0 \\ 0 & I_T \end{pmatrix},$$
  $B_i^{x}$ to the curve
  $$d\mapsto \begin{pmatrix} I_S & 0 \\ 0 & B_i^{xd} \end{pmatrix},$$
  and 
  $$u(V)=\begin{pmatrix} I_S & V \\ 0 & I_T \end{pmatrix}$$
  to the curve
  $$d\mapsto \begin{pmatrix} I_S & dV \\ 0 & I_T \end{pmatrix},$$
  where
  in all cases, $d$ ranges from $0$ to $1$.  Let $c\ge
  \max\{\ell(A_i),\ell(B_i)\}$.  Then the word $A_i^x u(v_i\otimes w)
  A_i^{-x}$ represents the matrix $u(\lambda_i^xv_i\otimes w)$ and
  corresponds to a curve of length at most $2cx+\|v_i\|_2\| w\|_2$
  connecting $I$ and $u(\lambda_i^xv_i\otimes w)$.  Similarly,
  if $t\ge 2$, then $B_i^{-x} u(v\otimes w_i) B_i^{x}$ has length
  at most $2cx+\|v_i\|_2\|w\|_2$ and connects $I$ and
  $u(\mu_i^xv\otimes w_i).$

  If $V\in \R^S\otimes \R^T$, then 
  $$V=\sum_{i,j} x_{ij}v_i\otimes w_j$$
  for some $x_{ij}\in \R$.  Let 
  $$l_{i}(x)= \begin{cases} \lceil \log_{\lambda_i} |x|\rceil &\text{ if $|x|>1$,} \\
    0&\text{ if $|x|\le 1$,}
  \end{cases}$$
  and define
  $$\gamma_{ij}(x)= A_i^{l_{i}(x)} u\biggl(\frac{x}{\lambda_i^{l_{i}(x)}} v_i\otimes w_j\biggr) A_i^{-l_{i}(x)}.$$
  Note that $|x/\lambda_i^{l_{i}(x)}|\le 1$.  Let
  $$\short{u}(V):=\prod_{i,j} \gamma_{ij}(x_{ij}).$$
  Then $\short{u}(V)$ represents $u(V)$ and there
  is a $c'$ such that 
  $$\ell(\short{u}(V))\le c'(1+\log \|V\|_2)$$
  for all $V$.  

  If $i\in S$ and $j\in T$, then $e_{ij}(x)=u(x z_i\otimes z_j)\in
  H_{S,T}$.  If $x\in \Z$, then we can apply Lemma~\ref{lem:approx} to
  approximate $\short{u}(x z_i\otimes z_j)$ by a word
  $\short{e}_{ij;S,T}(x)\in \Sigma^*$ which represents $e_{ij}(x)$ and
  whose length is $O(\log |x|)$.  In general, changing $S$ and $T$
  will change $\short{e}_{ij;S,T}(x)$ drastically, but later, we will
  prove that if $i\in S,S'$ and $j\in T,T'$, and $S$ and $S'$ satisfy
  some mild conditions, then $\short{e}_{ij;S,T}(x)$ and
  $\short{e}_{ij;S',T'}(x)$ are connected by a homotopy of area
  $O((\log |x|)^2)$.  Because of this, the choice of $S$ and $T$ is
  largely irrelevant.  Thus, for each $(i,j)$, we choose a $d\not\in
  \{i,j\}$ and let
  $$\short{e}_{ij}(x)=\short{e}_{ij;\{i,d\},\{j\}}(x).$$
\end{proof}

\begin{proof}[Proof of Lemma~\ref{lem:edgeChar}]
  If $e=(x,y)$ is an interior edge of $\tau$, it is in the boundary of
  two 2-cells; call these $\Delta$ and $\Delta'$.  By
  Corollary~\ref{cor:triBounds}, there is a $c$ depending only on $p$
  such that if $g=\bar{f}_0(x)^{-1}\bar{f}_0(y)\in \Gamma$ and
  $g_{ij}$ is the $(i,j)$-coefficient of $g$, then
  \begin{align*}
    &g\in P_{\Delta}(\Z)\cup P_{\Delta'}(\Z) \\
    &  |g_{ij}|<c  \text{\quad if $(i,j)\in \chi(M_{P_{\Delta}})\cup \chi(M_{P_{\Delta'}})$} \\
    &  \|g\|_{\infty}<c e^{c\ell}
  \end{align*}
  The last inequality follows from the
  fact that $d_{\cE}([I]_\cE,f(x))\le \ell$.  Note that $P_{\Delta}\cap P_{\Delta'}$ is parabolic.

  We express $g$ as a word in $\Sigma^*$ as follows.  Let $g=nm$,
  where $n\in N_{P_{\Delta}\cap P_{\Delta'}}(\Z)$ and $m\in
  M_{P_{\Delta}\cap P_{\Delta'}}(\Z)$.  Then $\|m\|_{\infty}< c,$ and there is a $c'$
  depending on $p$ such that $\|m\|_2 < c'$ and $\|m^{-1}\|_2< c'$.  Therefore,
  $$\|n\|_{\infty}\le \|gm^{-1}\|_{2}\le p^2c' c e^{c\ell}$$
  and if $(i,j)\in \chi(M_{P_{\Delta}})\cup \chi(M_{P_{\Delta'}})$, then $|n_{ij}|<pc'$.

  Since $n$ is a unipotent matrix, we can write $n$ as a product
  $$n=\prod_{(i,j)\in \chi(N_{P_{\Delta}\cap P_{\Delta'}})} e_{ij}(n_{ij})$$ 
  for an appropriate ordering of $\chi(N_{P_{\Delta}\cap
    P_{\Delta'}})$.  We can replace the terms corresponding to large
  coefficients with shortcuts.  Let
  $$w_1=\prod_{(i,j)\in \chi(N_{P_{\Delta}\cap P_{\Delta'}})} \begin{cases}
    e_{ij}^{n_{ij}} & \text{if $(i,j)\in \chi(M_{P_{\Delta}})\cup \chi(M_{P_{\Delta'}})$,}\\
    \short{e}_{ij}(n_{ij}) & \text{otherwise.}
  \end{cases}$$ This represents $n$ and has length $O(\ell)$.

  Finally, there is a $c''$ depending only on $p$ such that we
  can write $m$ as a product $w_2\in (\Sigma\cap M_{P_{\Delta}\cap
    P_{\Delta'}})^*$ of no more than $c''$ generators of
  $M_{P_{\Delta}\cap P_{\Delta'}}$.  Let
  $$\bar{f}_1(e)=w_1w_2\in \Sigma^*.$$
  This satisfies the conditions of the lemma for both $\Delta$ and
  $\Delta'$

  If $e$ is on the boundary of $\tau$ and $e$ is an edge of $\Delta$,
  then $P_\Delta=G$, and since $N_G=\{I\}$, there is a $c$ such that
  $d_{\Gamma}(\bar{f}_0(x),\bar{f}_0(y))<c$.  We can take $w_e$ to be
  a geodesic word representing $\bar{f}_0(x)^{-1}\bar{f}_0(y)$.
\end{proof}

We then construct $\bar{f}_1$ by defining $\bar{f}_1|_e$ to be the
curve corresponding to $w_e$.  Note that $\bar{f}_1|_{\partial\tau}$
differs from the original $w$ by only a bounded distance.  In
particular, there is an annulus in $K_\Gamma$ whose boundary curves
are $w$ and $\bar{f}_1|_{\partial\tau}$ and which has area $O(\ell)$.

\section{Filling the 2-skeleton}\label{sec:filling}
In the previous section, we reduced the problem of filling $\alpha$ to
the problem of filling the curves $\bar{f}_1(\partial\Delta)$, where
$\Delta$ ranges over all $2$-cells of $\tau$.  Each of these curves is
a product of a bounded number of elements of $\Sigma$ and a bounded
number of shortcuts $\short{e}_{ij}(x)$.  In this section, we will
describe methods for filling such curves.  The key to many of these
methods is the group $H_{S,T}$ from Section~\ref{sec:normalform}, which
we used to construct $\short{e}_{ij}$.  This group has two key
properties.  First, when either $S$ or $T$ is large enough, then $H_{S,T}$
has quadratic Dehn function; this is a special case of a theorem of de
Cornulier and Tessera.  Second, when both $S$ and $T$ are sufficiently
large, $H_{S,T}$ contains multiple ways to shorten elementary
matrices.  A good choice of shortening makes it possible to fill many
discs, including discs corresponding to the Steinberg relations.

We first state a special case of a theorem of de Cornulier and Tessera:
\begin{thm}[\cite{dCpersonal}]\label{thm:HDehn}
  If $s\ge 3$ or $t\ge 3$, then $H_{S,T}$ has quadratic Dehn function.
\end{thm}

The quadratic Dehn function will let us switch between different
shortenings.  Say $\#S\ge 3$, $\#T\ge 2$, and let $A_i\in \SL(S,\Z)$,
$B_i\in \SL(T,\Z)$, $v_i\in \R^S$, and $w_i\in \R^T$ be as in
Section~\ref{sec:normalform}.  Then we can express $u(x v_i\otimes w_j)$
either as $A_i^k u(v_i\otimes w_j) A_i^{-k}$ or as $B_j^{-l}
u(v_i\otimes w_j) B_j^{l}$.  In the following lemma, we switch between
these representations to find fillings for words representing
conjugates of $\short{u}(V)$.  Let $\Sigma_S:=\Sigma\cap \SL(S,\Z)$
and $\Sigma_T:=\Sigma\cap \SL(T,\Z)$.  These are generating sets for
$\SL(S,\Z)$ and $\SL(T,\Z)$.
\begin{lem} \label{lem:xiConj} If $\#S\ge 3$ and $\#T\ge 2$ or vice
  versa, there is an $\epsilon>0$ and a $c>0$ such that if $\gamma$ is
  a word in $(\Sigma_S\cup \Sigma_T)^*$ representing $(M,N)\in
  \SL(S,\Z) \times \SL(T,\Z)$, then
  $$\delta_{\cE(\epsilon)}([\gamma \short{u}(V)\gamma^{-1}]_\cE,[\short{u}(MVN^{-1})]_\cE)= c(\ell(\gamma)+\log{(\|V\|_2+2)})^2.$$
\end{lem}
\begin{proof}
  Let $\omega:=\gamma
  \short{u}(V)\gamma^{-1}]_\cE\short{u}(MVN^{-1})^{-1}$; this is a
  closed curve in $G$.

  We first consider the case that $V=x v_i\otimes w_j$ and $\gamma\in
  \Sigma_T^*$.  In this case, $M=I$; and $\gamma \short{u}(V)\gamma^{-1}$ and
  $\short{u}(VN^{-1})$ are both words in the group
  \begin{align*}
    F  &:=\left\{\begin{pmatrix}\prod_i A_i^{x_i} & V \\
      0                      & D \end{pmatrix} \middle|\; x_i\in \R, D\in \SL(T,\Z), V\in \R^S\otimes \R^T\right\} \\
    &=(\R^{s-1}\times \SL(T,\Z) )\ltimes (\R^S\otimes \R^T).
  \end{align*}
  This group is generated by 
  $$\Sigma_F:=\{A_i^x\mid x\in \R\}\cup \{u(V)\mid V\in \R^S\otimes \R^T\}\cup \Sigma_T.$$
  Let $\epsilon \le 1/2$ be sufficiently small that $H_{S,T}\subset
  G(\epsilon)$.  Since $G(\epsilon)$ is contractible and $F\subset
  G(\epsilon)$, words in $\Sigma_F^*$ correspond to curves in
  $G(\epsilon)$.  We will show that
  $$\delta_{G(\epsilon)}(\gamma \short{u}(V)\gamma^{-1},\short{u}(VN^{-1}))\le O(\ell(\omega)^2).$$

  Words in $\Sigma_F^*$ satisfy certain relations which correspond to
  discs in $G(\epsilon)$.  In particular, note that if $\sigma\in
  \Sigma_T$, $|x|\le 1$, and $\|W\|_2\le 1$, then 
  \begin{equation}\label{eq:commute}[\sigma, A_k^x]\end{equation}
  and
  \begin{equation}\label{eq:conj}\sigma u(W)\sigma^{-1}u(W\sigma^{-1})^{-1}  \end{equation}
  are both closed curves of bounded length.  Since $G(\epsilon)$ is
  contractible, their filling areas are bounded, and we can think of them as ``relations'' in $F$.

  Let $C=\log_{\min_k\{\lambda_k\}} (p+1)$, and let
  $z=C\ell(\gamma)+l_i(x)$.  This choice of $z$ ensures that
  $$\|\lambda_i^{-z} V N\|_2\le 1.$$
  Indeed, it ensures that if $d_{\SL(T,\Z)}(I,N')\le \ell(\gamma)$, then
  $$\|\lambda_i^{-z} V N'\|_2\le 1.$$
  Furthermore, $z=O(\ell(\omega))$.

  We will construct a homotopy which lies in $G(\epsilon)$ and goes
  through the stages
  \begin{align*}
    \omega_1&=\gamma \short{u}(V)\gamma^{-1} \\
    \omega_2&=\gamma A_i^{z} u(\lambda_i^{-z} V) A_i^{-z} \gamma^{-1}  \\
    \omega_3&= A_i^{z} \gamma u(\lambda_i^{-z} V) \gamma^{-1} A_i^{-z} \\
    \omega_4&= A_i^{z} u(\lambda_i^{-z} V N^{-1}) A_i^{-z} \\
    \omega_5&= \short{u}(VN^{-1}).
  \end{align*}
  Each stage is a word in $\Sigma_F^*$ and so corresponds to a curve
  in $G(\epsilon)$.  

  We can construct a homotopy between $\omega_1$ and $\omega_2$ and
  between $\omega_4$ and $\omega_5$ using Thm.~\ref{thm:HDehn}.  We
  need to construct homotopies between $\omega_2$ and $\omega_3$ and
  between $\omega_3$ and $\omega_4$.  

  We can transform $\omega_2$ to $\omega_3$ by applying
  \eqref{eq:commute} at most $O(\ell(\omega)^2)$ times.  This
  corresponds to a homotopy with area $O(\ell(\omega)^2)$.  Similarly,
  we can transform $\omega_3$ to $\omega_4$ by applying
  \eqref{eq:conj} at most $O(\ell(\omega))$ times, corresponding to a
  homotopy of area $O(\ell(\omega))$.  Combining all of these
  homotopies, we find that
  $$\delta_{G(\epsilon)}(\gamma \short{u}(V)\gamma^{-1},\short{u}(VN^{-1}))\le O(\ell(\omega)^2).$$
  as desired.

  We can use this case to generalize to the case $V=\sum_{i,j}x_{ij}
  v_i\otimes w_j$ and $\gamma\in \Sigma_T^*$.  By applying the case to
  each term of $\short{u}(V)$, we obtain a homotopy of area
  $O(\ell(\omega)^2)$ from $\gamma \short{u}(V)\gamma^{-1}$ to
  $$\prod_{i,j}\short{u}(x_{ij} v_i\otimes w_j N^{-1}).$$
  This is a curve in $H_{S,T}$ of length $O(\ell(\omega))$ which connects $I$
  and $u(VN^{-1})$.  By Thm.~\ref{thm:HDehn}, there is a homotopy between
  this curve and $\short{u}(VN^{-1})$ of area $O(\ell(\omega)^2)$.

  When $\gamma\in \Sigma_S^*$, we instead let $F$ be the group
  \begin{align*}
    F&:=\left\{\begin{pmatrix}D & V \\
        0 & \prod_i B_i^{x_i} \end{pmatrix} \middle|\; x_i\in \R, D\in
      \SL(S,\Z), V\in \R^S\otimes \R^T\right\} \\
    &=(\SL(S,\Z)\times  \R^{t-1} )\ltimes (\R^S\otimes \R^T).
  \end{align*}
  Here, $\short{u}(V)$ is not
  a word in $F$, but since $\#T\ge 2$, we can replace the $A_i$ with
  the $B_i$ in the construction of $\short{u}(V)$.  This results in
  shortcuts $\short{u}'(V)$ in the alphabet
  $$\{B_i^x\mid x\in \R\}\cup\{u(V)\mid V\in \R^S\otimes \R^T\}.$$
  These are curves in $H_{S,T}$ which represent $u(V)$ and have length
  $O(\log\|V\|_2)$, so by Thm.~\ref{thm:HDehn}, there is a homotopy of
  area $O((\log\|V\|_2)^2)$ between $\short{u}'(V)$ and $\short{u}(V)$.

  The argument for $\gamma\in \Sigma_S^*$ shows that
  $$\delta_{G(\epsilon)}(\gamma \short{u}'(V)\gamma^{-1},\short{u}'(MV))=O(\ell(\omega)^2).$$
  Replacing $\short{u}'(V)$ with $\short{u}(V)$ and $\short{u}'(MV)$
  with $\short{u}(MV)$ adds area $O(\ell(\omega)^2)$, so 
  $$\delta_{G(\epsilon)}(\gamma \short{u}(V)\gamma^{-1},\short{u}(MV))=O(\ell(\omega)^2).$$

  If $\gamma\in (\Sigma_S\cup \Sigma_T)^*$, and $\gamma_S\in
  \Sigma_S^*$ and $\gamma_T\in \Sigma_T^*$ are the words obtained by
  deleting all the letters in $\Sigma_T$ and $\Sigma_S$ respectively,
  then $\delta_G(\gamma,\gamma_S\gamma_T)=O(\ell(\omega)^2)$.  
  We can construct a homotopy from $\gamma
  \short{u}(V)\gamma^{-1}$ to $\short{u}(MVN^{-1}))$ going through the
  steps
  \begin{align*}
    \gamma \short{u}(V)\gamma^{-1} & \to \gamma_S \gamma_T \short{u}(V)\gamma_T^{-1} \gamma_S^{-1}\\
    &\to \gamma_S \short{u}(VN^{-1}) \gamma_S^{-1}\\
    &\to \short{u}(MVN^{-1}).
  \end{align*}
  This homotopy has area $O(\ell(\omega)^2)$.
\end{proof}

Recall that $\short{e}_{ij;S,T}(x)$ is an approximation of a curve
$\short{u}(x z_i\otimes z_j)$; we write this curve as
$\short{u}_{S,T}(x z_i\otimes z_j)$ to distinguish curves in
different solvable subgroups.

\begin{lem}\label{lem:shortEquiv}
  If $p\ge 5$, $i\in S,S'$ and $j\in T,T'$, where $2\le \#S,\#S'\le  p-2$, then
  $$\delta_{\Gamma}(\short{e}_{ij;S,T}(x), \short{e}_{ij;S',T'}(x))=O((\log |x|)^2).$$
\end{lem}
\begin{proof}
  \noindent  Case 1:
  Let $V=x z_i\otimes z_j$.  We first consider the case that $S=S'$.  Both
  $\short{u}_{S,T}(V)$ and $\short{u}_{S',T'}(V)$ are curves in
  $H_{S,S^c}$ for $S^c$ the complement of $S$.  Since $k\ge 5$, Thm.~\ref{thm:HDehn} states that $H_{S,S^c}$ has quadratic Dehn function,
  so the lemma follows.  In particular,
  $$\delta_{\Gamma}(\short{e}_{ij;S,T}(x), \short{e}_{ij;S,\{j\}}(x))=O((\log |x|)^2).$$

  \noindent  Case 2: Let $S\subset S'$, $\#S'\ge 3$, $T\subset T'$,
  and $\#T'\ge 2$. Let $\{A_i\}$ be as in the definition of $H_{S,T}$,
  with eigenvectors $v_i$ and let $\{A'_i\}\in \SL(S',\Z)$ be the set
  of independent commuting matrices used in defining $H_{S',T'}$.
  Recall that $\short{u}_{S,T}(V)$ is the concatenation of curves
  $\gamma_i$ of the form
  $$A_i^{c_i} u(x_i v_i\otimes z_j) A_i^{-c_i}$$
  where $c_i\in \Z$ and $|x_i|\le 1$.  Since $A_i\in
  \SL(S,\Z)\subset \SL(S',\Z)$, each of these curves satisfies the hypotheses of
  Lemma~\ref{lem:xiConj} for $S'$ and $T'$, and so there is a homotopy
  of area $O((\log|x|)^2)$ between $\gamma_i$ and 
  $$\short{u}_{S',T'}(\lambda_i^{c_i} x_i v_i\otimes z_j).$$
  Each of these curves lie in $H_{S',T'}$, and since
  $\short{u}_{S',T'}(V)$ also lies in $H_{S',T'}$ and $H_{S',T'}$ has
  quadratic Dehn function,
  $$\delta_{\Gamma}(\short{e}_{S,T}(V),\short{e}_{S',T'}(V))=O((\log
  |x|)^2).$$

  Combining these two cases proves the lemma.  First, we construct a
  homotopy between $\short{e}_{S,T}(V)$ and a word of the form
  $\short{e}_{\{i,d\},\{j\}}(V)$.  If $\#S=2$, we can use case 1.
  Otherwise, let $d\in S$ be such
  that $d\ne i$.  We can construct a homotopy going through the stages
  $$\short{e}_{S,T}(V)\to \short{e}_{S,S^c}(V)\to \short{e}_{\{i,d\},\{j\}}(V).$$
  The second step is an application of case 2, possible because
  $\{i,d\}\subset S$, $\#S\ge3$, and $\{j\}\subset S^c$.  

  Similarly, we can construct a homotopy between $\short{e}_{S,T}(V)$
  and a word of the form $\short{e}_{\{i,d'\},\{j\}}(V)$.  If $d=d'$,
  we're done.  Otherwise, we can use case 2 to construct homotopies
  between each word and $\short{e}_{\{i,d,d'\},\{i,d,d'\}^c}(V)$.
\end{proof}

Using these lemmas, we can give fillings for a wide variety of curves; note that (\ref{lem:infPres:add})--(\ref{lem:infPres:commute}) are versions of the Steinberg relations.
\begin{lem}\label{lem:infPres} If $p\ge 5$ and $x,y\in \Z\setminus
  \{0\}$, then
  \begin{enumerate}
  \item \label{lem:infPres:add}
    If $1\le i,j\le p$ and $i\ne j$, then
    $$\delta_{\Gamma}(\short{e}_{ij}(x)\short{e}_{ij}(y),\short{e}_{ij}(x+y))=O((\log |x|+\log |y|)^2).$$
    In particular,
    $$\delta_{\Gamma}(\short{e}_{ij}(x)\short{e}_{ij}(-x))=O((\log |x|)^2).$$
  \item \label{lem:infPres:multiply}
    If $1\le i,j,k\le p$ and $i\ne j\ne k$, then
    $$\delta_{\Gamma}([\short{e}_{ij}(x),\short{e}_{jk}(y)],\short{e}_{ik}(xy))= O((\log |x|+\log |y|)^2).$$
  \item \label{lem:infPres:commute}
    If $1\le i,j,k,l\le p$, $i\ne l$, and $j\ne k$
    $$\delta_{\Gamma}([\short{e}_{ij}(x),\short{e}_{kl}(y)])=O((\log |x|+\log |y|)^2).$$
  \item \label{lem:infPres:swap}
    Let $1\le i,j,k,l\le p$, $i\ne j$, and $k\ne l$, and 
    $$s_{ij}=e_{ji}^{-1}e_{ij}e_{ji}^{-1},$$
    so that $s_{ij}$ represents
    $$\begin{pmatrix} 0 & 1 \\ -1 & 0
    \end{pmatrix}\in\SL(\{i,j\},\Z).$$  Then
    $$\delta_{\Gamma}(s_{ij} \short{e}_{kl}(x)
    s^{-1}_{ij},\short{e}_{\sigma(k)\sigma(l)}(\tau(k,l)x))=O( (\log
    |x|+\log |y|)^2),$$
    where $\sigma$ is the permutation switching
    $i$ and $j$, and $\tau(k,l)=-1$ if $k=i$ or $l=i$ and $1$
    otherwise.
  \item \label{lem:infPres:diag} If $b=\diagmat(b_1,\dots,b_p)$, then 
    $$\delta_{\Gamma}(b \short{e}_{ij}(x) b^{-1},\short{e}_{ij}(b_i b_j x)(\tau(k,l)x))=O( \log |x|^2).$$  \end{enumerate}
\end{lem}
\begin{proof}

  For part \ref{lem:infPres:add}, note that
  $$\short{e}_{ij}(x)\short{e}_{ij}(y)\short{e}_{ij}(x+y)^{-1}$$
  is
  within bounded distance of a closed curve in $H_{\{j\}^c,\{j\}}$ of
  length $O(\log|x|)$.  Thus part \ref{lem:infPres:add} of the lemma follows from
  Thm.~\ref{thm:HDehn}.

  For part \ref{lem:infPres:multiply}, let $d\not \in \{i,j,k\}$ and
  let $S=\{i,j,d\}$, so that $\short{e}_{ij;\{i,d\},\{j\}}(x)$ is a
  word in $\SL(S,\Z)$.  We construct a homotopy going through
  the stages
  \begin{align*}
    &  [\short{e}_{ij}(x),\short{e}_{jk}(y)]\short{e}_{ik}(xy)^{-1} &\\
    &  [\short{e}_{ij;\{i,d\},\{j\}}(x),\short{u}_{S,\{k\}}(y z_{j}\otimes z_{k})]\short{e}_{ik;S,\{k\}}(xy)^{-1} & \text{by Lem.~\ref{lem:shortEquiv}}\\
    &  \short{u}_{S,\{k\}}((xy z_i+y z_{j})\otimes z_{k})\short{u}_{S,\{k\}}(y z_{j}\otimes z_{k})^{-1}\short{e}_{ik;S,\{k\}}(xy z_{i}\otimes z_k)^{-1}& \text{by Lem.~\ref{lem:xiConj}}\\
    &  \emptyword & \text{by Thm.~\ref{thm:HDehn}}
  \end{align*}
  All these homotopies have area $O((\log |x|+\log |y|)^2)$.
  
  For part \ref{lem:infPres:commute}, we let $S=\{i,j,d\}$, $T=\{k,l\}$, and use the same
  techniques to construct a homotopy going through the stages
  \begin{align*}
    & [\short{e}_{ij}(x),\short{e}_{kl}(y)])\\
    & [\short{e}_{ij;S,T}(x),\short{e}_{kl;S,T}(y)] & \text{by Lem.~\ref{lem:shortEquiv}}\\
    & \emptyword& \text{by Thm.~\ref{thm:HDehn}}
  \end{align*}
  This homotopy has area $O((\log |x|+\log |y|)^2)$.

  Part \ref{lem:infPres:swap} breaks into several cases depending on $k$ and $l$.  When
  $i,j,k,$ and $l$ are distinct, the result follows from part \ref{lem:infPres:commute},
  since $s_{ij}=e_{ji}^{-1}e_{ij}e_{ji}^{-1}$, and we can use part
  \ref{lem:infPres:commute} to commute each letter past $\short{e}_{kl}(x)$.  If $k=i$ and
  $l\ne j$, let $d,d'\not\in \{i,j,l\}$, $d\ne d'$, and let
  $S=\{i,j,d\}$ and $T=\{l,d'\}$.  There is a homotopy from
  $$s_{ij} \short{e}_{il}(x) s^{-1}_{ij}\short{e}_{jl}(-x)^{-1}$$
  to
  $$s_{ij} \short{u}_{S,T}(x z_i\otimes z_l) s^{-1}_{ij}\short{e}_{jl}(-x z_j\otimes z_l)$$
  of area $O( (\log |x|)^2),$ and since $s_{ij}\in \Sigma_S^*$, the
  proposition follows by an application of Lemma \ref{lem:xiConj}.  A
  similar argument applies to the cases $k=j$ and $l\ne i$; $k\ne i$
  and $l= j$; and $k\ne j$ and $l= i$.

  If $(k,l)=(i,j)$, let $d, d'\not \in \{i,j\}$.
  There is a homotopy going through the stages
  \begin{align*}
&    s_{ij} \short{e}_{ij}(x) s^{-1}_{ij} & \\
&    s_{ij} [e_{id},\short{e}_{dj}(x)] s^{-1}_{ij}& \text{ by part \ref{lem:infPres:multiply}}\\
&    [s_{ij}e_{id}s^{-1}_{ij},s_{ij}\short{e}_{dj}(x)s^{-1}_{ij}]& \text{ by free insertion}\\
&    [e_{jd}^{-1},\short{e}_{di}(x)] & \text{ by previous cases}\\
&    \short{e}_{jd}(-x) & \text{ by part \ref{lem:infPres:multiply}}
  \end{align*}
  and this homotopy has area $O( (\log |x|)^2)$.  One can treat the case $(k,l)=(j,i)$ the same way.

  Since any diagonal matrix in $\Gamma$ is the product of at most
  $p$ elements $s_{ij}$, part \ref{lem:infPres:diag} follows from part \ref{lem:infPres:swap}.
\end{proof}

This lemma allows us to fill shortenings of curves in nilpotent subgroups of
$\Gamma$ efficiently.  
\begin{lemma} \label{lem:shortNP}
  Let $P=U(S_1,\dots, S_s)\in \cP$, let $w_i= \short{e}_{a_ib_i}(x_i)$ and let
  $w=w_1\dots w_d$
  for some $(a_i,b_i)\in \chi(N_P)$.  Let $h=\max\{\log |x_i|,1\}$.  If
  $w$ represents the identity, then $\delta_G(w)=O(d^3h^2)$.
\end{lemma}
\begin{proof}
  We first describe a normal form for elements of $N_P$.  Let
  $$\chi_k(N_P)=\{(a,b)\mid a\in S_k, (a,b)\in \chi(N_P)\}.$$
  The set $\{e_{ab}\mid (a,b)\in \chi_k(N_P)\}$ generates an
  abelian subgroup of $\Gamma$.  If $n\in N_P$, let $n_{ab}$ be the
  $(a,b)$-coefficient of $n$ and let
  $$\kappa_q(n)=\prod_{(a,b)\in \chi_q(N_P)} \short{e}_{ab}(n_{ab}).$$
  Let
  $$\nu_P(n)=\kappa_s(n)\kappa_{s-1}(n)\dots \kappa_1(n)$$
  This is a word representing $n$, and it has length $O(\log
  \|n\|_2)$.

  Let $n_i\in \Gamma$ be the element represented by $w_1\dots w_i$.
  There is a $c$ such that $\log \|n_i\|_2\le chd$.  The words
  $\nu_P(n_i)$ connect the identity to points on $w$, so we can
  fill $w$ by filling the wedges $\nu_P(n_{i-1})w_{i}
  \nu_P(n_{i})^{-1}$; we consider this filling as a homotopy between $\nu_P(n_{i-1})w_{i}$ and $\nu_P(n_{i})$.
  Note that if $a_i\in S_k$, then 
  $$\kappa_s(n_{i-1})\dots\kappa_{k+1}(n_{i-1})=\kappa_s(n_{i})\dots\kappa_{k+1}(n_{i}),$$
  so it suffices to transform
  $$\kappa_k(n_{i-1})\dots\kappa_{1}(n_{i-1}) w_i\to \kappa_k(n_{i})\dots\kappa_{1}(n_{i}).$$

  We can use parts \ref{lem:infPres:multiply} and \ref{lem:infPres:commute} of
  Lemma~\ref{lem:infPres} to move $w_i$ to the left.
  That is, we repeatedly replace subwords of the form
  $\short{e}_{ab}(x)w_i$
  with
  $w_i\short{e}_{ab}(x)$
  if $b\ne a_i$ and with $w_i\short{e}_{ab_i}(x x_i)\short{e}_{ab}(x)$
  if $b= a_i$.  We always have $a\in S_j$
  for some $j\le k$, so $a<b_i$.

  Each step has cost $O((\log |x|+\log |x_i|)^2)$.  Since $\log|x|\le
  \log \|n_i\|_2\le chd$, this is $O(h^2d^2)$.  We repeat this process
  until we have moved $w_i$ to the left end of the word, which takes
  at most $p^2$ steps and has total cost $O(h^2d^2)$.  The result is a
  word of the form $\kappa'_k\dots \kappa'_1$ where $\kappa'_q$ is a
  product of words of the form $\short{e}_{ab}(x)$ for $(a,b)\in
  \chi_q(N_P)$.  Furthermore, the $\kappa'_q$ are obtained from the
  $\kappa_q(n_i)$ by inserting at most $p^2$ additional words in all (at
  most one word is added in each step, in addition to the original
  $w_i$).

  Since the elements represented by the
  terms of $\kappa'_q$ all commute, we can use parts \ref{lem:infPres:add} and \ref{lem:infPres:commute} of
  Lemma~\ref{lem:infPres} to rearrange the terms in each $\kappa'_q$ and
  transform $\kappa'_k\dots \kappa'_1$ into $\kappa_k(n_{i})\dots\kappa_{1}(n_{i})$.  This takes at
  most $4p^4$ applications of part \ref{lem:infPres:commute} and at most $2p^2$ applications
  of part \ref{lem:infPres:add}, each of which has cost $O(h^2d^2)$.  Thus
  $$\delta_{\Gamma}(\kappa'_k\dots \kappa'_1, \kappa_k(n_{i})\dots\kappa_{1}(n_{i}))=O(h^2d^2),$$
  and so
  $$\delta_{\Gamma}(\nu_P(n_{i-1})w_{i} \nu_P(n_{i})^{-1})=O(h^2d^2).$$
  To fill $w$, we need to fill $d$ such wedges, so $\delta_{\Gamma}(w)=O(h^2d^3).$
\end{proof}
In particular, if $d$ is fixed, then $\delta_{\Gamma}(w)=O(h^2)$.

Finally, we use these tools to fill the curves that occur as $\bar{f}_1(\partial\Delta)$.
\begin{lemma}
  If $\Delta$ is a 2-cell in $\tau$,
  $$\delta_{K_\Gamma}(\bar{f}_1(\partial\Delta))=O(\ell^2).$$
\end{lemma}
\begin{proof}
  By Lemma~\ref{lem:edgeChar}, there is a $c$ depending only on $p$
  such that we can write the word corresponding to
  $\bar{f}_1(\partial\Delta)$ as $g=g_1\dots g_d$, where $d\le c$ and
  each $g_i$ is either an element of $\Sigma\cap M_{P_\Delta}$ or a
  word $\short{e}_{ab}(x)$ where $(a,b)\in \chi(N_{P_\Delta})$ and
  $|x|\le c e^{c\ell}$.

  Let $x_i\in P$ be the element represented by
  $g_1\dots g_i$; by the hypotheses, there is a $c'$ independent of
  $\alpha$ such that $\|x_i\|_2\le c' e^{c'\ell}$. Let
  $x_i=m_in_i$ for some $m_i\in M_P$ and $n_i\in N_P$.  Then
  $d_{\Gamma}(I,m_i)\le c$, and there is a $c''$ independent of
  $\alpha$ such that $\|n_i\|_2 \le c'' e^{c'' \ell}$.

  Let $\gamma_i$ be a geodesic word representing $m_i$, and let
  $w_i=\gamma_i\nu_P(n_i)$.  The $w_i$ are words of length
  $O(\ell(\alpha))$ connecting points on $g$ to the identity, and
  we can get a filling of $g$ by filling the wedges
  $w_ig_{i+1}w_{i+1}^{-1}$.

  The filling depends on $g_{i+1}$.  If $g_{i+1}\in \Sigma_{M_P}$,
  then
  $$w_ig_{i+1}w_{i+1}^{-1}=\gamma_i \nu_P(n_i) g_{i+1} \nu_P(n_{i+1})^{-1}\gamma_{i+1}^{-1},$$
  and $g_{i+1}^{-1}n_ig_{i+1}=n_{i+1}$.
  Lemma~\ref{lem:infPres} allows us to move $g_{i+1}$ past the
  individual terms of $\nu_P(n_i)$, using $O(\ell(\alpha)^2)$ steps.
  After this, we have a word of the form
  $$\gamma_i g_{i+1} h_1\dots h_k \nu_P(n_{i+1})^{-1}\gamma_{i+1}^{-1},$$
  where $h_i=\short{e}_{a_ib_i}(x_i)$ for some $(a_i,b_i)\in
  \chi(N_P)$,  $|x_i| \le c'' e^{\ell(\alpha)}$, and $k\le p^2$.  By
  Lemma~\ref{lem:shortNP}, $h_1\dots h_k \nu_P(n_{i+1})^{-1}$ can be
  reduced to the trivial word at cost $O(\ell^2)$.  This
  leaves us with the word $\gamma_i g_{i+1} \gamma_{i+1}^{-1}$; this
  has length at most $2c+1$ and can be reduced to the trivial word at
  bounded cost.

  If $g_{i+1}=\short{e}_{ab}(x)$ for $(a,b)\in \chi(N_P)$, then
  $\gamma_i=\gamma_{i+1}$, and $\nu_P(n_i) g_{i+1}
  \nu_P(n_{i+1})^{-1}$ represents the identity.  This satisfies the
  hypotheses of Lemma~\ref{lem:shortNP}, and can be reduced to the
  trivial word at cost $O(\ell(\alpha)^2)$.  This leaves $\gamma_i\gamma_{i+1}^{-1}$; as before, this has length at most $2c$ and can thus be reduced to the trivial word at bounded cost.  

  Thus the cost of filling each wedge is $O(\ell^2)$.  Since
  there are at most $c$ wedges, the cost of filling $w$ is
  $O(\ell^2)$.
\end{proof}
Since there are $O(\ell^2)$ such 2-cells to fill, we can fill
$\bar{f}_1(\partial\tau)$ with area $O(\ell^4)$.  Furthermore,
$\bar{f}_1(\partial\tau)$ is a bounded distance from $w$ in
$K_\Gamma$, so
$$\delta_\Gamma(w)\le \delta_{K_\Gamma}(w,\bar{f}_1(\partial\tau))+\delta_{K_\Gamma}(\bar{f}_1(\partial\tau))=O(\ell^4).$$
This proves Theorem~\ref{thm:mainthm}.

\def\cprime{$'$}
\providecommand{\bysame}{\leavevmode\hbox to3em{\hrulefill}\thinspace}
\providecommand{\MR}{\relax\ifhmode\unskip\space\fi MR }
\providecommand{\MRhref}[2]{%
  \href{http://www.ams.org/mathscinet-getitem?mr=#1}{#2}
}
\providecommand{\href}[2]{#2}

\end{document}